\documentclass[10pt,a4paper, final, twosided]{article}

 \usepackage[a4paper,left=35mm,right=35mm,top=30mm,bottom=30mm,marginpar=25mm]{geometry}
\usepackage{color}
\usepackage{mathtools}
\usepackage{amsfonts,hyperref,url}
\usepackage{amssymb}
\usepackage{amsmath}
\usepackage{amssymb}
\usepackage{amsthm}
\usepackage[latin1]{inputenc}
\usepackage{eurosym}
\usepackage{graphicx}
\usepackage{epsfig}
\usepackage{hyperref}
\usepackage{dsfont}
\usepackage{appendix}

 \usepackage{array}

\allowdisplaybreaks

\newcommand{\ba}{\begin{eqnarray}}
\newcommand{\ea}{\end{eqnarray}}

\makeindex

\renewcommand{\div}{\operatorname{div}}

\newcommand{\Rr}{{\mathbb{R}}}

\newcommand{\Tt}{{\mathbb{T}}}

\newcommand{\td}{{\mathbb{T}^d}}

\newcommand{\epsi}{\varepsilon}

\theoremstyle{plain}

\newtheorem{theorem}{Theorem}

\newtheorem{proposition}[theorem]{Proposition}

\newtheorem*{theorem*}{Theorem}

\theoremstyle{defi}
\newtheorem{definition}[theorem]{Definition}

\newtheorem{stp}{Step}

\theoremstyle{remexample}
\newtheorem{remark}[theorem]{Remark}

\newtheorem{teo}[theorem]{Theorem}

\theoremstyle{ass}

\begin{document}
	
	\title{On a Mean Field Optimal Control Problem}
	
	\author{Jos\'e A. Carrillo\footnote{
		Department of Mathematics, Imperial College London, London SW7 2AZ, UK. E-mail: carrillo@imperial.ac.uk},
		Edgard A. Pimentel\footnote{
		Department of Mathematics, Pontifical Catholic University of Rio de Janeiro, Rio de Janeiro, 22495-900, Brazil. E-mail: pimentel@puc-rio.br},
		Vardan K. Voskanyan\footnote{
			Centro de Matem\'atica da Universidade de Coimbra, Apartado 3008, EC Santa Cruz, 3001 - 501 Coimbra, Portugal. E-mail: vartanvos@gmail.com }}
	
	\date{} 

	\maketitle
	\begin{abstract}
\noindent In this paper we consider a mean field optimal control problem with an aggregation-diffusion constraint, where agents interact through a potential, in the presence of a Gaussian noise term. Our analysis focuses on a PDE system coupling a Hamilton-Jacobi and a Fokker-Planck equation, describing the optimal control aspect of the problem and the evolution of the population of agents, respectively. The main contribution of the paper is a result on the existence of solutions for the aforementioned system. We notice this model is in close connection with the theory of mean-field games systems. However, a distinctive feature concerns the nonlocal character of the interaction; it affects the drift term in the Fokker-Planck equation as well as the Hamiltonian of the system, leading to new difficulties to be addressed. 

\bigskip


\bigskip


\end{abstract}

\maketitle

\section{Introduction}
The multi-agent framework has been used to model swarming or collective behavior in a number of applications. Those include animal herding or flocking \cite{PhysRevLett.96.104302,Camazine_etal,CFTVreview,HCH,lukeman}, cell movement \cite{HP,PBSG,Pbook}, cell adhesion \cite{GC,BDZ,BCDPZ}, alloy clustering \cite{nano1}, opinion formation \cite{DMPW}, robotics \cite{TFYE}, among others. 

Controlling these collective-behavior models, both at the microscopic and macroscopic levels, has recently become a very popular research direction \cite{CFPT13,FoSo,CFPT15,BFK,FPR,PRT,ACFK}. Of particular interest is the mean-field limit of a system of $N$-agents interacting through a potential $W$ in presence of Gaussian noise. It leads to the control problem of a density $\rho$ solving an aggregation-diffusion equation of the form:
\begin{equation}\label{Eq.control}
	\begin{cases}
		\frac{\partial \rho}{\partial t}=\div\left( (\nabla W\star \rho)\rho+F\rho\right)+\Delta \rho& \;\;\;\;\;\mbox{in}\;\;\;\;\;\Tt^d\times(0,\infty)\\
		\rho(x, 0)=\rho_0(x)&\;\;\;\;\;\mbox{in}\;\;\;\;\;\Tt^d
	\end{cases}	
\end{equation}
where the control $F$ is chosen to minimize the functional:
\begin{equation}\label{eq_control1}
\inf_{\rho, F}\left\{E(F)+\int_{\Tt^d}\phi_T(x)\rho(x,T) dx\right\}
\end{equation}
with
\begin{equation}\label{eq_control2}
	E(F)=\int_0^T\!\!\!\int_{\Tt^d}\left[L(x,\rho)+\frac{|F|^2} 2\right]\rho dx dt.
\end{equation}

A distinctive feature in \eqref{Eq.control} regards the drift term. In addition to the optimal control component $F$, it depends on $\rho$ itself through the interactions driven by $W$. Perhaps more involving is the fact that such a dependence is nonlocal; this character of the drift term poses important mathematical difficulties to the program developed in this paper.

Typical potentials used in applications are radially symmetric $W(x)=w(|x|)$ in the whole space chosen to be repulsive in the short range and attractive in the long range. Some examples include the Morse or power-law potentials
\[
 w(r)=-C_Ae^{-\frac{r}{l_A}}+C_Re^{-\frac {r}{l_R}}
\]
or 
\[
 w(r)=\frac{r^a}{a}- \frac{r^b}{b}\,,
\]
for well-prepared parameters $Cl^d<1$, with $C:=C_A/C_R$ and $l:=l_A/l_R$, and $a>b>-d$ respectively. See \cite{PhysRevLett.96.104302,predict,CHM} and the references therein. Many other biological applications use finite range potentials, i.e. compactly supported, but with local behaviors near the origin, similar to the Morse or power-laws potentials above. We refer the reader to \cite{CP,BDZ,BCDPZ,CCS}, to name just a few. Notice that these potentials can be used in the periodic setting $x\in\Tt^d$, see \cite{CP,BDZ,BCDPZ}.

The optimal control problem in \eqref{eq_control1}-\eqref{eq_control2} can be written as
\[
\inf_{\rho, F\in \eqref{Eq.control}}\sup_{\phi}\left\{E(F)+\int_{\Tt^d}\phi_T(x)\rho(x,T) dx- \int_{0}^{T}\!\!\!\int_{\Tt^d}\phi\left[\frac{\partial \rho}{\partial t}-\div\left( (\nabla W\star \rho)\rho+F\rho\right)-\Delta \rho\right]\right\},
\]
reminiscent of optimal transport problems \cite{Bre}. By changing the order of the infimum and the supremum, we obtain the dual problem
\[
\sup_{\phi}\inf_{(\rho, F)\in \eqref{Eq.control}}\left\{ \int_0^T\!\!\!\int_{\Tt^d}\left[L(x,\rho)+\frac{|F|^2} 2+\frac{\partial \phi}{\partial t}-(\nabla W\star \rho)\cdot \nabla \phi-F\cdot \nabla \phi+\Delta\phi\right]\rho(x,t) dx dt\right\}.
\]
At least heuristically, this optimization problem leads to a system of PDEs of the form
\begin{equation}\label{Eq.main-secondorder}
\begin{cases}
\displaystyle -\phi_t+\frac{|\nabla \phi|^2}{2}+(\nabla W\star \rho)\cdot \nabla \phi+\nabla W\star (\rho \nabla \phi)-U(x,\rho)=\Delta \phi&\;\;\;\mbox{in}\;\;\;\Tt^d\times(0,T)\\[2mm]
\displaystyle \rho_t=\div\left( (\nabla W\star \rho)\rho+\rho\nabla \phi\right)+\Delta\rho&\;\;\;\mbox{in}\;\;\;\Tt^d\times(0,T)\\[2mm]
\phi(x,T)=\phi_T(x),\, \rho(x,0)=\rho_0(x)&\;\;\;\mbox{in}\;\;\;\Tt^d,
\end{cases}
\end{equation}
where $U(x,\rho)=L(x,\rho)+\frac{\delta L}{\delta \rho}(x,\rho)\rho$, $F=-\nabla\phi$, and 
$$
\nabla W\star (\rho \nabla \phi) = \sum_{i=1}^d \frac{\partial W}{\partial x_i}\star \left(\rho \frac{\partial \phi}{\partial x_i}\right).
$$
It is worthy noticing that the system in \eqref{Eq.main-secondorder} can be regarded as a first order condition associated with \eqref{eq_control1}-\eqref{eq_control2}. These first-order optimality conditions were obtained both formally and rigorously in \cite{ACFK}. We pose the problem in the $d$-dimensional torus $\Tt^d$ to focus on the main difficulties related to the regularity of the Hamilton-Jacobi equation in the system \eqref{Eq.main-secondorder}. 

\begin{remark}\label{rem:convtor}
Throughout the paper we use the convention to identify functions on the torus $\mathbb{T}^d$ with 1-periodic functions on the cube $Q^d=[-\frac 1 2, \frac 1 2]^d.$ Hence, since the function $|x|$ has periodic boundary data on that cube it can be considered as function on the torus, and therefore radial potentials can be considered in the torus as well. Furthermore, the convolution $\nabla W \star \rho$ is defined by:
\[
\nabla W \star \rho(x):=\int_{Q^d} \nabla W(x-y)\rho(y) dy=\int_{Q^d} \nabla W(y)\rho(x-y) dy,
\]
where the values $\nabla W(x-y),\,\rho(x-y)$ are defined by periodicity.
\end{remark}

The contribution of the present paper concerns the existence and regularity of the solutions to \eqref{Eq.main-secondorder}. Our main result is the following:

\begin{teo}[Existence of solutions]
	\label{teo-main}
	Let the following assumptions hold:
	\begin{enumerate}
	        \item[{\bf (A1)}] Regularity of the potential: $W\in W^{2, \infty}(\Tt^d)=C^{1,1}(\Tt^d)$.
	        
		\item[{\bf (A2)}] Regularity of the coupling: $U:\Tt^d\times L^2(\Tt^d)\to \Rr$ is uniformly bounded in $x$ in $\mathcal{C}^2$ norm and continuous in $m$ in $L^2$ norm, i.e., there exists a constant $C>0$ such that
			\[
			\|U(\cdot, m)\|_{\mathcal{C}^2},\ \|\phi_T\|_{\mathcal{C}^2}\leq C,\,\forall m\in L^2(\Tt^d),
			\]
		for every $m \in L^2(\Tt^d)$, and
		\[
		|U(x,m_1)-U(x,m_2)|\leq C\|m_1-m_2\|_{L^2(\Tt^d)},
		\]
		for every $x\in \Tt^d\times(0,T)$ and for all $m_1,\,m_2\in L^2(\Tt^d)$.

		\item[{\bf (A3)}] Initial-terminal boundary conditions: $\rho_0\in L^2(\Tt^d)$ with $\int \rho_0=1$, $\rho_0\geq 0$, and $\phi_T\in\mathcal{C}^1(\Tt^d)$.
	\end{enumerate}
	Then, there exists a solution $(\phi, \rho)$ to the optimal control system \eqref{Eq.main-secondorder} with $\rho\in L^\infty(0,T;L^2(\Tt^d))$ and	$\phi \in \mathcal{C}^{1,2}(\Tt^d\times [0,T] )$. 
\end{teo}

\begin{remark}
Typical potentials with power law behavior satisfying our assumptions are $$W(x)\simeq\frac{|x|^a}a - \frac{|x|^b}{b}\text{, with }a, b\geq 2,$$
to be understood in the sense of Remark \ref{rem:convtor}.
\end{remark}

The system in \eqref{Eq.main-secondorder} relates to the mean-field games (MFG, for short) introduced by Jean-Michel Lasry and Pierre-Louis Lions \cite{ll1, ll2, ll3, ll4}, and, independently, by Minyi Huang, Roland P. Malham{\'e}, and Peter E. Caines \cite{Caines2,Caines1}. Indeed, it couples a Hamilton-Jacobi equation describing an optimization problem, with the Fokker-Planck equation accounting for the evolution of the population. An interesting feature of \eqref{Eq.main-secondorder} regards the lack of an adjoint structure, which yields further difficulties from the regularity viewpoint.

In the recent years, the MFG theory developed in a number of directions. The existence and uniqueness of solutions is the topic of \cite{cgbt,cd2,CLLP,cllp13,cargra,pim4,pim2,pim3,marcir1,marcir2,marcir3,marcir4}, whereas the study of numerical methods is the object of \cite{achdou2013finite,MR2928376,CDY,DY}, to name just a few. In \cite{carpor,benone,bentwo,cardel1,delaruegalera,gangboswiech} the authors examine the master equation. Applications of the MFG framework to social sciences can be found in \cite{moll,bmoll1,bmoll2,bmoll3}. We also refer the reader to the monographs \cite{cardaliaguet,bensoussan,pim1} and the lists of references therein.

An important toy-model in the context of time-dependent mean-field games has the form
\begin{equation*}
\begin{cases}
-u_t-\nu\Delta u+H(D_xu, x, \theta)=0&\;\;\;\;\;\mbox{in}\;\;\;\;\;\Tt^d\times(0,T)\\
\theta_t -\nu\Delta \theta-\div(D_pH(D_xu, x, \theta)\theta)=0&\;\;\;\;\;\mbox{in}\;\;\;\;\;\Tt^d\times(0,T)\\
u(x, T)=\psi(x, \theta(T)),\quad
\theta(x, 0)=\theta_0&\;\;\;\;\;\mbox{in}\;\;\;\;\;\Tt^d.
\end{cases}
\end{equation*}
This class of systems consists of a Hamilton-Jacobi equation coupled with a Fokker-Plank equation; the latter is the formal adjoint, in the $L^2$-sense, of the linearization of the former one. This (adjoint) structure was exploited in proving the existence of classical solutions \cite{GPatVrt, GPM1, GM, PV15} through the so-called non-linear adjoint methods, introduced by L.C. Evans \cite{E3}. The system in \eqref{Eq.main-secondorder} does not present this adjoint structure because of the additional term $\nabla W\star (\rho \nabla \phi)$, which substantially complicates the arguments. However, the non-linear adjoint method is still useful in proving the Lipschitz continuity of a solution $\phi$ to the Hamilton-Jacobi equation in \eqref{Eq.main-secondorder}. 

The nonlocal interaction among agents affects the PDE system accounting for the optimality conditions of the model. Here, the drift term in the Fokker-Planck equation becomes nonlocal. A similar phenomenon takes place in the Hamilton-Jacobi counterpart of the system: the Hamiltonian becomes nonlocal \emph{with respect to the gradient of the solutions}.

In this context, at least two genuine difficulties appear. First, the existence and uniqueness of the solutions to the Fokker-Planck equation becomes a nontrivial matter. Also the (uniform) compactness of the solutions to the Hamilton-Jacobi equation does not follow from the usual techniques. At first, one could resort to semiconvexity/semiconcavity properties to bypass this issue - however, we notice the equation falls short in preserving those conditions.

To circumvent the first question, we resort to an argument in \cite{Por}, combined with a fixed-point strategy performed in an appropriate space of measures. As for the compactness for the solutions to the Hamilton-Jacobi, we reason through a nonlocal $L^p$-regularity theory, together with (compact) embedding results. We believe our techniques are flexible enough to produce information on a larger class of problems.

The remainder of this article is structured as follows: in Section \ref{Sec:mainthm} we present a brief outline of the proof of Theorem \ref{teo-main}. Section \ref{sec_existfp} establishes the well-posedness for the Fokker-Planck equation in \eqref{Eq.main-secondorder}, whereas in Section \ref{Sec:Lip} we produce a number of a priori estimates for the solutions of the Hamilton-Jacobi. Finally, a section reporting the proof of Theorem \ref{teo-main} closes the paper.


\section{Set up and outline of the proof}\label{Sec:mainthm}

In what follows, let us introduce the general lines along which we establish Theorem \ref{teo-main}. 
Let us start with the definition of (weak) solution used in the paper.

\begin{definition}[Weak solution]\label{def_solution}
A pair $(\phi,\rho)$ is a solution to \eqref{Eq.main-secondorder} if
\begin{enumerate}
\item $\phi\in\mathcal{C}(\td\times(0,T))$ satisfies the first equation in \eqref{Eq.main-secondorder} in the viscosity sense;
\item $\rho\in L^\infty(0,T;L^2(\td))$ satisfies the second equation in \eqref{Eq.main-secondorder} in the sense of distributions.
\end{enumerate}
\end{definition}

Throughout the paper, we denote by  $\mathcal{C}^{1,2}(\td\times[0,T])$ the space of functions defined over $\td\times[0,T]$ that are of class $\mathcal{C}^1$ with respect to time and class $\mathcal{C}^2$ with respect to space. Similarly, for any $\alpha\in(0,1),$ $\mathcal{C}^{\alpha,1+\alpha}(\td\times[0,T])$ stands for the functions that are $\alpha-$H\H older continuous in time, $\mathcal{C}^1$ with respect to space with $\alpha-$H\H older continuous  derivative.

We establish Theorem \ref{teo-main} by using a fixed-point argument. The argument starts by taking an element $\phi$ in $\mathcal{C}^{0,1}(\Tt^d\times[0,T])$. To such a function $\phi\in\mathcal{C}^{0,1}(\Tt^d\times[0,T])$, we can assign the unique weak solution $\rho\in L^\infty(0,T;L^2(\td))$ to the Fokker-Planck equation
\begin{equation}\label{Eq.rhophi}
	\frac{\partial \rho}{\partial t}=\div\left( (\nabla W\star \rho)\rho+\nabla \phi\,\rho\right)+\Delta \rho,
\end{equation}
equipped with initial condition $\rho(x,0)=\rho_0(x)$ in the $d$-dimensional torus.
 
Then, standard results in the literature ensure the existence of a unique solution, $\Phi\in \mathcal{C}^{1,2}(\Tt^d\times[0,T])$ to
\begin{equation}\label{eq:HJB}
	\begin{cases}
		-\Phi_t+\frac{|\nabla\Phi|^2}{2}+(\nabla W\star \rho)\cdot \nabla \Phi+\nabla W\star (\rho \nabla \phi)=\Delta \Phi +U(x,\rho)&\;\;\mbox{in}\;\;\Tt^d\times(0,T)\\
		\Phi(x,T)=\phi_T(x)&\;\;\mbox{in}\;\;\Tt^d\times(0,T);
	\end{cases}	
\end{equation}
see, for instance, \cite[Section 3.2]{cardaliaguet}. This procedure induces a mapping $\phi\,\to\,\rho\,\to\,\Phi$, which we denote
$$\mathcal{F}\,:\,\cap_{\alpha\in (0,1)}\mathcal{C}^{\alpha,1+\alpha}(\td\times[0,T])\,\longrightarrow\, \cap_{\alpha\in (0,1)}\mathcal{C}^{\alpha,1+\alpha}(\td\times[0,T]),$$
and define as
\[
\mathcal{F}(\phi):=\Phi.
\]

To prove the existence of a solution $(\phi,\rho)$ to problem \eqref{Eq.main-secondorder} is tantamount to verify that the mapping $\mathcal{F}$ has a fixed point. To that end, we start with the definition of an appropriate subset $\mathcal{K}_{\alpha}\subset \mathcal{C}^{\alpha,1+\alpha}(\Tt^d\times[0,T])$. For some constants $A$, $B$, and $C>0$ to be determined further, we define
\[
\mathcal{K}_{\alpha}=\left\{\phi\in\mathcal{C}^{\alpha,1+\alpha}(\td\times[0,T])\,|\, \left\|\Phi\right\|_{\mathcal{C}^{\alpha,1+\alpha}}\leq C \mbox{ and } \|\Phi(\cdot,t)\|_{Lip}\leq Ae^{B(T-t)}\right\}.
\]
 \begin{remark}
	We remark here that the set $\mathcal{K}_{\alpha}$ is nonempty, since for any smooth function $\phi$ appropriately rescaled $\phi(\lambda x, \lambda t)$ is in $\mathcal{K}_{\alpha}$. Furthermore, the set of smooth function $\mathcal{C}^{\infty}(\td\times[0,T])$ is dense in $\mathcal{K}_{\alpha}$, this can be proved by a simple mollification argument.
\end{remark}
 
The first step towards the proof of Theorem \ref{teo-main} is a result on the well-posedness for the Fokker-Planck equation in \eqref{Eq.main-secondorder}.

\begin{proposition}[Well-posedness for the Fokker-Planck equation]\label{prop_fpeu}
Suppose the assumptions {\rm\bf (A1)-(A3)} hold and $\phi\in\mathcal{K}_{\alpha}$ is fixed. Then, there exists a unique solution to \eqref{eq_fp1}.
\end{proposition}

Once the existence and uniqueness for the Fokker-Planck equation \eqref{eq_fp1} is assured, we turn our attention to the compactness of the solutions to the Hamilton-Jacobi equation \eqref{eq:HJB}. We proceed with the following proposition.

\begin{proposition}
	\label{proposition.bounds}
	Let the assumptions {\rm\bf (A1)-(A3)} hold. Then, there exists a choice of constants $A$, $B$, $C$  depending only on $T$, $\phi_T$, $\rho_0$, $\|W\|_{W^{2,\infty}(\Tt^d)}$ and $U$, such that for any $\phi\in \mathcal{C}(\td\times[0,T])$, satisfying
	\[
	|\phi(x, t)-\phi(y,t)|\leq Ae^{B(T-t)}|x-y|,
	\]
	we have
	\begin{equation}\label{phiset}
	\|\Phi(\cdot,t)\|_{Lip}\leq Ae^{B(T-t)},\;\;\;\|\Phi\|_{\infty}\leq C,
	\nonumber
	\end{equation}
	for every solution $\Phi$ to \eqref{eq:HJB}.
\end{proposition}

From now on, the set $\mathcal{K}_{\alpha}$ is defined with the constants $A$ and $B$ given by Proposition \ref{phiset}. We now turn to fix $C$. The $L^p$-regularity theory for the Hamilton-Jacobi equation is pivotal in proving uniform bounds for the solutions in appropriate H\"older spaces. This is the content of the next proposition.

\begin{proposition}\label{proposition_c1alpha}
Let $\Phi$ be a solution to \eqref{eq:HJB} corresponding to $\phi\in \mathcal{K}_{\alpha}$. Suppose the assumptions {\rm\bf (A1)-(A3)} hold true. Then, $\Phi\in\mathcal{C}^{\alpha,1+\alpha}(\mathbb{T}^d\times[0,T])$ for some $\alpha\in(0,1)$. Moreover, there exists $C>0$ depending solely on $T$, $\phi_T$, $\rho_0$, $\|W\|_{\mathcal{C}^2(\Tt^d)}$ and $U$, such that
\[
	\left\|\Phi\right\|_{\mathcal{C}^{\alpha,1+\alpha}(\mathbb{T}^d\times[0,T])}\,\leq\,C.
\]
\end{proposition}

From now on, the set $\mathcal{K}_{\alpha}$ is fixed with the definition of $C$ given in the previous proposition. We detail next the proof of Proposition \ref{prop_fpeu}.


\section{Well-posedness for the Fokker-Planck equation}\label{sec_existfp}

In this section, we prove Proposition \ref{prop_fpeu} on the existence and uniqueness of solutions to
\begin{equation}\label{por1}
		\begin{cases}
			\rho_t\,=\,\div\left[\left(\left(\nabla W\star\rho\right)\,+\,\nabla \phi\right)\rho\right]\,+\,\Delta\rho&\;\;\;\;\;\mbox{in}\;\;\;\;\;\mathbb{T}^d\times[0,T]\\
			\rho(x,0)\,=\,\rho_0(x)&\;\;\;\;\;\mbox{in}\;\;\;\;\;\mathbb{T}^d,
		\end{cases}
\end{equation}
provided $W$ and $\phi$ are well-prepared. We reason through a fixed-point argument; the next proposition details the functional space we work. Before proceeding let us recall the 1-Wasserstein metric on the space of probability measures $\mathcal{P}(\Tt^d)$:
\begin{definition}[Wasserstein metric]
	The $1$-Wasserstein metric between probability measures $\mu, \nu\in\mathcal{P}(\Tt^d)$ is given by
	\[
	d_1(\mu,\,\nu):=\inf_{\pi\in\Gamma(\mu,\,\nu)}\int_{\Tt^d\times\Tt^d}|x-y|d\pi(x,y),
	\]
where $ \Gamma (\mu ,\nu )$ is the collection of all measures on $\Tt^d\times\Tt^d$ with marginals $ \mu $  and $\nu .$
\end{definition}
Recall that the $d_1$ distance in bounded sets is the same as the bounded Lipschitz distance defined via duality with respect to Lipschitz functions, see \cite{Vil}. 
\begin{proposition}\label{prop_fp1}Define the function
\[
	N(t):=\int_{\mathbb{T}^d}\left|(\nabla W\star\mu)+\nabla \phi\right|^2d\mu(x,t).
\]
Take $C>0$ to be determined later. Set 
\[
	\mathcal{M}:=\left\lbrace\mu\in\mathcal{C}([0,T],\mathcal{P}(\Tt^d)):d_1(\mu(s),\mu(t))\leq C|t-s|^\frac{1}{2},\left\|N\right\|_{L^\infty([0,T])}\leq C \right\rbrace.
\]
Then, $\mathcal{M}$ is a convex and compact subset of $\mathcal{C}([0,T],\mathcal{P}(\Tt^d))$.
\end{proposition}

The proof of the Proposition \ref{prop_fp1} follows along the same lines as in \cite[Lemma 5.7]{cardaliaguet}, except for minor modifications, and is omitted here.

In what follows, we define a mapping $\mathcal{T}:\mathcal{C}([0,T],\mathcal{P}(\Tt^d))\to\mathcal{C}([0,T],\mathcal{P}(\Tt^d))$. Let $\phi\in\mathcal{K}_{\alpha}$ be fixed and take $\rho^0\in\mathcal{M}$. Solve
\begin{equation}\label{eq_fp1}
	\rho_t\,=\,\div\left[\left(\left(\nabla W\star\rho^0\right)\,+\,\nabla \phi\right)\rho\right]\,+\,\Delta\rho\;\;\;\;\;\mbox{in}\;\;\;\;\;\mathbb{T}^d\times[0,T]
\end{equation}
equipped with initial condition $\rho(x,0)=\rho_0$. By \cite[Lemma 3.3]{cardaliaguet}, we know that there exists a solution $\rho^1$ to \eqref{eq_fp1}. Notice, that 
\[
	\left|\left(\nabla W\star\rho^0\right)\,+\,\nabla \phi\right|^2\rho^1\,\in\,L^1(\mathbb{T}^d\times(0,T)).
\]
In fact, 
\[
	\int_0^T\int_{\mathbb{T}^d}\left|\left(\nabla W\star\rho^0\right)\,+\,\nabla \phi\right|^2d\rho^1(x,t)\,\leq\,C(W,\phi,T).
\]
Therefore, it follows from \cite[Theorem 1.1]{Por} that \eqref{eq_fp1} has at most one solution, which is precisely $\rho^1$. As a consequence, we define
\begin{equation}\label{eq_fp2}
	\mathcal{T}(\rho^0)\,:=\,\rho^1.
\end{equation}
Next we examine properties of the mapping $\mathcal{T}$ which are related to the fixed-point arguments presented further in this section.

\begin{proposition}\label{prop_fp2}
Let $\mathcal{T}:\mathcal{C}([0,T],\mathcal{P}(\Tt^d))\to\mathcal{C}([0,T],\mathcal{P}(\Tt^d))$ be defined as in \eqref{eq_fp2}. Then, $\mathcal{T}$ maps $\mathcal{M}$ into itself. In addition, $\mathcal{T}$ is a continuous mapping when restricted to the set $\mathcal{M}$.
\end{proposition}
\begin{proof}
First, we verify the inclusion $\mathcal{T}(\mathcal{M})\subset\mathcal{M}$. Let $\mu\in\mathcal{M}$ and $\sigma:=\mathcal{T}(\mu)$. We have
\[
	d_1(\sigma(t),\sigma(s))\leq \mathbb{E}(X_t,X_s),
\]
where 
\begin{equation}\label{eq_fp3}
	\begin{cases}
		dX_t\,=\,\left[(\nabla W\star\mu)\,+\,\nabla \phi\right]dt\,+\,\sqrt{2}IdW_t\\
		X(0)\,=\,X_0,
	\end{cases}
\end{equation}
and $W_t$ is a $d$-dimensional Brownian motion. Because the vector field $(\nabla W\star\mu)\,+\,\nabla \phi$ is H\"older continuous with respect to space, there exists a unique solution to \eqref{eq_fp3}; see, for example, \cite{chadru1}. Hence, using the Lipschitz continuity of $W$ and $\phi$, we obtain
\begin{align*}
	d_1(\sigma(t),\sigma(s))\,&\leq\,\mathbb{E}\left[\int_s^t\left|(\nabla W\star\mu)\,+\,\nabla \phi\right|d\tau\,+\,\sqrt{2}\left|W_t\,-\,W_s\right|\right]\\
	&\leq\,C_1\left|t\,-\,s\right|^\frac{1}{2},
\end{align*}
with $C_1=C_1(W,\phi,T)$. Moreover, once more from the Lipschitz continuity of $W$ and $\phi$, we infer
\[
	\int_0^T\int_{\mathbb{T}^d}\left|\nabla W\star\sigma\,+\,\nabla \phi\right|^2d\sigma(x,t)\,\leq\,C\int_0^T\int_{\mathbb{T}^d}\sigma(x,t)dxdt\,\leq\,C_2,
\]
where $C_2=C_2(W,\phi,T)$. 
Finally, by choosing the constant $C>0$ in the definition of $\mathcal{M}$ as
$
	C\,:=\,\max\left\lbrace C_1,\,C_2\right\rbrace,
$
we have $\mathcal{T}(\mathcal{M})\subset\mathcal{M}$.

It remains to prove that $\mathcal{T}$ is continuous. However, it follows from standard results in stability theory for the Fokker-Planck equation in the presence of H\"older-continuous drift terms \cite{Por}. 
\end{proof}

\begin{remark}
We notice that the constant $C>0$ in the definition of $\mathcal{M}$ depends solely on $W$ and $\phi$.
\end{remark}

We close this section with the proof of Proposition \ref{prop_fpeu}.

\begin{proof}[Proof of Proposition \ref{prop_fpeu}]
We start by noticing that $\mathcal{T}$ has a fixed point. In fact, Proposition \ref{prop_fp1} ensures that $\mathcal{M}$ is a convex and compact subset of $\mathcal{C}([0,T],\mathcal{P}_1)$. From Proposition \ref{prop_fp2} we infer that $\mathcal{T}(\mathcal{M})\subset\mathcal{M}$ and that this mapping is continuous. Therefore, the Schauder's Fixed Point Theorem yields the existence of a fixed point for $\mathcal{T}$. The uniqueness is ensured by a straightforward application of \cite[Theorem 1.1]{Por} and the proposition is established.
\end{proof}

\begin{proposition}[Propagation of $L^2$ regularity]
	Let the  assumptions of Proposition \ref{prop_fpeu} hold, 
	then $\rho \in L^{\infty}( 0,T; L^2 (\Tt^d))$.
\end{proposition}

\begin{proof}
	We take a standard  radially symmetric mollifier $\varphi_{\epsi}\in C^{\infty}(\Tt^d)$ and a smooth function $\psi\in C^{\infty}_c([0,T)\times \Tt^d)$, and  consider $\varphi_{\epsi}\star\psi(t, \cdot)$ as test function for the weak solution  $\rho$:
	\[
	\int_0^T\int_{\Tt^d}(\varphi_{\epsi}\star\psi_t\,-\,\Delta \varphi_{\epsi}\star\psi\,+\, b D\varphi_{\epsi}\star\psi)\rho\,dxdt\,=\,
	\int_{\Tt^d} \rho_0 \varphi_{\epsi}\star\psi(0) dx,
	\]
	where $b(x,t):=\nabla W\star\rho\,+\,\nabla \phi$. Using the identity
	\[
	\int_{\Tt^d}(\varphi_{\epsi}\star g)(x)\, f(x)\,dx =\int_{\Tt^d}g(x)\, (\varphi_{\epsi}\star f)(x) \,dx,
	\]
	we get
	\[
	\int_0^T\int_{\Tt^d}(\psi_t\,-\,\Delta \psi) (\varphi_{\epsi}\star\rho) \,+\, D\psi  (\varphi_{\epsi}\star (b\rho)) \,dxdt\,=\,
	\int_{\Tt^d} \varphi_{\epsi}\star\psi\, \rho_0 dx\,,
	\]
	implying $\varphi_{\epsi}\star\rho$ is  weak solution to
	\[
	\begin{cases}
	(\varphi_{\epsi}\star\rho)_t\,=\,\div\left[\varphi_{\epsi}\star (b\rho))\right]\,+\,\Delta	(\varphi_{\epsi}\star\rho)&\;\;\;\;\;\mbox{in}\;\;\;\;\;\mathbb{T}^d\times[0,T]\\
	\rho(x,0)\,=\,	\varphi_{\epsi}\star\rho_0(x)&\;\;\;\;\;\mbox{in}\;\;\;\;\;\mathbb{T}^d.
	\end{cases}
	\]
	Since $t\mapsto \rho\, dx$ is continuous with respect to $d_1$ and $\nabla W$ is in $C^1(\Tt^d)$, then $\nabla W \ast \rho$ is continuous in time and is in $C^1(\Tt^d)$. Since $\nabla \phi \in C^\alpha(\Tt^d)$ the vector field $b$ is continuous in time and is in $C^\alpha(\Tt^d)$. This implies that the first term of the right hand side of the above equation, $\div\left[\varphi_{\epsi}\star (b\rho)\right]$, is a bounded function, hence $\rho_{\epsi}:=\varphi_{\epsi}\star\rho$ is a classical solution due to classical regularity of the heat equation.
	
	Multiplying it by $\rho_{\epsi}$ and integrating in space we get by integration by parts
\begin{align*}
\frac{d}{dt}{\|\rho_{\epsi}(\cdot,t)\|_{L^2(\Tt^d)}^2}
=&\,2\int_{\Tt^d}\left[-|\nabla \rho_{\epsi}|^2\,-\,\varphi_{\epsi}\star (b\rho)\cdot \nabla \rho_{\epsi}\right] dx\\
\leq&
\frac 1 2\int_{\Tt^d} (\varphi_{\epsi}\star (b\rho))^2dx
\leq \frac 1 2 \|b\|_{\infty}  \int_{\Tt^d} (\varphi_{\epsi}\star\rho)^2dx.
\end{align*}
Since $\|b\|_{\infty}\leq \|\nabla W\|_{\infty}\,+\,\|\nabla \phi\|_{\infty}<+\infty$, Gronwall's inequality implies
\[
\|\rho_{\epsi}(\cdot,t)\|_{L^2(\Tt^d)}^2\leq e^{\frac 1 2 C t} \|\rho_{\epsi}(\cdot,0)\|_{L^2(\Tt^d)}^2.
\]
By taking the limit $\epsi\to 0$ we finally obtain $\|\rho\|_{L^{\infty}([0,T], L^2(\Tt^d) )}<+\infty.$
\end{proof}


\section{Lipschitz continuity and H\"older regularity}\label{Sec:Lip}

In this section we detail the proofs of Propositions \ref{proposition.bounds} and \ref{proposition_c1alpha}. Under the assumption $W\in W^{2,\infty}$ the coefficients of the equation \eqref{eq:HJB} are $C^1$ in the $x$ variable. As a consequence, we have that $\Phi$ is $C^2$ in the $x$ variable. This fact can be verified through a Hopf-Cole transformation; see e.g. \cite[Section 3.2]{cardaliaguet}.

We resort to the nonlinear adjoint method, introduced by L. C. Evans in \cite{Evans.adjoint}. For a detailed discussion on the application of this method to the theory of mean field games, we refer the reader to \cite{GPatVrt,pim1}.

Start by fixing $(x_0,t_0)\in\Tt^d\times[0,T]$ and consider two adjoint variables $\eta$, $\zeta$ given by the equations
	\begin{equation*}\label{aux1}{\textit{Optimal Flow:\quad}}
	\begin{cases}
	\eta_t\,=\,\div\left( (\nabla W\star \rho)\eta\,+\,\nabla \Phi\eta\right)\,+\,\Delta \eta&\;\;\;\;\;\mbox{in}\;\;\;\;\;\td\times(0,T)\\
	\eta(\cdot,t_0)\,=\,\delta_{x_0}&\;\;\;\;\;\mbox{in}\;\;\;\;\;\td
	\end{cases}
	\end{equation*}
	and
	\begin{equation}\label{aux0}{\textit{Zero velocity Flow:\quad}}
	\begin{cases}
	\zeta_t\,=\,\div\left( (\nabla W\star \rho)\zeta\right)\,+\,\Delta \zeta&\;\;\;\;\;\mbox{in}\;\;\;\;\;\td\times(0,T)\\
	\zeta(\cdot,t_0)\,=\,\delta_{x_0}&\;\;\;\;\;\mbox{in}\;\;\;\;\;\td.
	\end{cases}
	\end{equation}
	The first one represents the evolution of the distribution of the optimal trajectories in the stochastic optimal control problem associated with the first equation in \eqref{Eq.main-secondorder}. The second one accounts for the evolution of the distribution of the trajectory corresponding to the zero velocity. Note that $$\int_{\Tt^d}\eta(x,t)dx\,=\,\int_{\Tt^d}\zeta(x,t)dx\,=\,1$$ and $\eta,\,\zeta\geq 0$.
	
	Note that in general $\eta$ and $\zeta$ are measure valued. However since the vector fields $(\nabla W\star \rho)\,+\,\nabla \Phi$ and $(\nabla W\star \rho)$ are $C^1$ in space, for times $t>t_0$, $\eta$ and $\zeta$ can be viewed as functions. Alternatively, to make the rest of the arguments rigorous one can approximate the delta distributions $\delta_{x_0}$ by smooth approximations of unity and pass to the limit at the end. Below we will avoid these complications for simplicity.

\begin{proof}[Proof of Proposition \ref{proposition.bounds}]
We split the proof of the proposition in several steps.

\bigskip

 \begin{stp}\label{Stepl2}
 	\textbf{An estimates on $\int_{t_0}^T\int_{\Tt^d}\frac{|\nabla\Phi|^2}{2}\eta dxdt$:}
 \end{stp}
 We integrate \eqref{eq:HJB} against the measure $\zeta$ and use the equation \eqref{aux0} in the distributional sense to obtain
	$$
	\Phi(x_0,t_0)\!-\!\int_{\Tt^d}\phi_T\zeta_T dx+\!\!\int_{t_0}^T\!\!\!\int_{\Tt^d}\!\!\frac{|\nabla\Phi|^2}{2}\zeta dxdt+\!\!\int_{t_0}^T\!\!\!\int_{\Tt^d}\!\!\nabla W\star (\rho \nabla \phi)\zeta dx dt=\!\int_{t_0}^T\!\!\!\int_{\Tt^d}\! U\zeta dxdt.
	$$
	Subsequently, we infer that
	\begin{equation}\label{Eq.Phiupper}
	\Phi(x_0,t_0)\leq \|\nabla W\|_{\infty}\int_{t_0}^T\|\nabla \phi(\cdot, t)\|_{\infty}dt+(T-t_0)\|U\|_{\infty}+\|\phi_T\|_{\infty}.
	\end{equation}
Similarly, multiplying \eqref{eq:HJB} by $\eta$ and using \eqref{aux1} in the distributional sense, we have
	\[
	\Phi(x_0,t_0)-\int_{\Tt^d}\phi_T\eta_T dx-\int_{t_0}^T\!\!\!\int_{\Tt^d}\frac{|\nabla\Phi|^2}{2}\eta dxdt+\int_{t_0}^T\!\!\!\int_{\Tt^d}\nabla W\star (\rho \nabla \phi)\eta dx dt=\int_{t_0}^T\!\!\!
	\int_{\Tt^d}U\eta dxdt.
	\]
	Thus, we get
	\begin{equation}\label{Eq.Philow}
	\int_{t_0}^T\!\!\!\int_{\Tt^d}\frac{|\nabla\Phi|^2}{2}\eta dxdt\leq \Phi(x_0,t_0)+\|\phi_T\|_{\infty}+\|\nabla W\|_{\infty}\int_{t_0}^T\|\nabla \phi(\cdot, t)\|_{\infty}dt+(T-t_0)\|U\|_{\infty}\,.
	\end{equation}
	From \eqref{Eq.Phiupper} and \eqref{Eq.Philow} we deduce
	\ba\label{Eq.Philower}
	\int_{t_0}^T\!\!\!\int_{\Tt^d}\frac{|\nabla\Phi|^2}{2}\eta dxdt	\leq 2\|\phi_T\|_{\infty}+2\|\nabla W\|_{\infty}\int_{t_0}^T\|\nabla \phi(\cdot, t)\|_{\infty}dt+2(T-t_0)\|U\|_{\infty}.
	\ea
\

\begin{stp}\label{Steplip}
	\textbf{Lipschitz continuity in space:}
\end{stp}
	Fix $\xi\in\Rr^d$ and let $u=\nabla\Phi\cdot\xi:=\Phi_{\xi}$, then differentiating the equation for $\Phi$, we get	
	\[
	-u_t+\nabla\Phi\cdot \nabla u+(\nabla W_{\xi} \star \rho)\cdot \nabla \Phi+(\nabla W\star \rho)\cdot \nabla u+\nabla W_{\xi}\star (\rho \nabla \phi)=\Delta u+U_{\xi}.
	\]
As above, integrating the equation for $\Phi$ against $\zeta$ and using  integration by parts, we obtain
	\begin{equation*}
	u(x_0,t_0)-\!\int_{\Tt^d}\!\!\!\nabla\phi_T\eta_T dx+\int_{t_0}^T\!\!\!\int_{\Tt^d}\!\!(\nabla W_{\xi} \star \rho)\cdot \nabla \Phi\eta dxdt+\int_{t_0}^T\!\!\!\int_{\Tt^d}\!\!\!\nabla W_{\xi} \star (\rho \nabla \phi)\eta dx dt=\!\!\int_{t_0}^T\!\!\!\int_{\Tt^d}\!\! U_{\xi}\eta dxdt.
	\end{equation*}
	Hence
	\begin{equation*}
	|u(x_0,t_0)|\leq\|\nabla\phi_T\|_{\infty}+\|\nabla^2 W\|_{\infty}\int_{t_0}^T\!\!\!\left(\int_{\Tt^d} |\nabla \Phi|\eta dx+ \|\nabla \phi(\cdot, t)\|_{\infty}\right)dt+(T-t_0)\|\nabla U\|_{\infty}+C.
	\end{equation*}
	Since 
	\[\int_{t_0}^T\!\!\!\int_{\Tt^d} |\nabla \Phi|\eta dxdt\leq\frac 1 2 \int_{t_0}^T\!\!\!\int_{\Tt^d} [1+|\nabla \Phi|^2]\eta dxdt,\]
	we use \eqref{Eq.Philower} to produce
	\begin{equation*}
	\begin{split}
	|u(x_0,t_0)|\leq A+B\int_{t_0}^T\|\nabla \phi(\cdot, s)\|_{\infty}ds
	\end{split}
	\end{equation*}
with
$$
	A=\|\nabla\phi_T\|_{\infty}+\|\nabla^2 W\|_{\infty}\left(\frac1 2 T+2\|\phi_T\|_{\infty}
	+2T\|U\|_{\infty}\right)+T\|\nabla U\|_{\infty}+C
$$
and $B=\|\nabla^2 W\|_{\infty}\left(2\|\nabla W\|_{\infty}+1\right)$.
Since $x_0$ and $t_0$ were arbitrary, we have proved 
	\[
	\|\nabla\Phi(\cdot,t)\|_{\infty}\leq A+B\int_{t}^T\|\nabla \phi(\cdot, s)\|_{\infty}ds,\quad \forall t\in[0,T].
	\]
	Now it is easy to check that
	\[
	\|\nabla \phi(\cdot,t)\|_{\infty}\leq Ae^{B(T-t)}\implies \|\nabla \Phi(\cdot,t)\|_{\infty}\leq  Ae^{B(T-t)}.
	\]
This together with \eqref{Eq.Phiupper} and \eqref{Eq.Philow} yields $\|\Phi\|_{\infty}\leq C_1,$ where $C_1$ depends on $\|W\|_{\mathcal{C}^2}$,$\|U\|_{\mathcal{C}^2}$,$\|\phi_T\|_{\mathcal{C}^1}$, $A$ and $B$.
\end{proof}


Next we present the proof of Proposition \ref{proposition_c1alpha}.

\begin{proof}[Proof of Proposition \ref{proposition_c1alpha}]
We observe that \eqref{eq:HJB} can be written as
\[
	-\Phi_t\,-\,\Delta\Phi\,=\,f(x,t)\,\in\,L^\infty(\mathbb{T}^d\times[0,T]),
\]
with a bound on $\|f\|_{L^\infty}$ depending only on $\|W\|_{\mathcal{C}^2}$,$\|U\|_{\mathcal{C}^2}$,$\|\phi_T\|_{\mathcal{C}^1}$, $A$ and $B$.
Therefore, by standard elliptic regularity theory, we have
\[
	\Phi_t,\,\nabla^2 \Phi\,\in\,L^p(\mathbb{T}^d\times[0,T]),
\]
for every $p>1$, see for example \cite[Theorem D.1. and Remark D.1.4]{lionsbook}. Morrey's Embedding Theorem implies the result.
\end{proof}



\section{Proof of Theorem \ref{teo-main}}\label{Sec:mapcont}

In this section we present the proof of Theorem \ref{teo-main}. In fact, we show that $\mathcal{F}$ restricted to $\mathcal{K}_{\alpha}$ is a continuous mapping. This fact builds upon Propositions \ref{prop_fpeu}-\ref{proposition_c1alpha} to cover the existence of a fixed point of $\mathcal{F}$ under the scope of the Schauder's Fixed-Point Theorem \cite[Theorem 11.1]{GilTru}.

\begin{theorem*}[Schauder's Fixed-Point Theorem]
Let $K$ be a compact convex set in a Banach space $\mathbb{B}$ and let $T$ be a
continuous mapping of $K$ into itself. Then $T$ has a fixed point, that is, $Tx\,=\,x$ for
some $x\in\mathbb{B}$. 
\end{theorem*}

\begin{proof}[Proof of Theorem \ref{teo-main}]
Take $0<\beta<\alpha<1$, it is obvious that $\mathcal{K}_{\alpha}$ is a closed and convex set which is compact in $\mathcal{C}^{\beta, 1+\beta}(\td\times[0,T])$. Its compactness follows from Propositions \ref{proposition.bounds} and \ref{proposition_c1alpha}, together with the Arzel\`a-Ascoli Theorem.  The fact that $\mathcal{F}({\mathcal{K}_{\alpha}})\subset \mathcal{K}_{\alpha}$ follows from Proposition \ref{proposition.bounds}. 

It remains to prove that the map $\mathcal{F}\big|_{\mathcal{K}_{\alpha}}$ is continuous in $\mathcal{C}^{\beta, 1+\beta}(\td\times[0,T])$. For that we need to show that for any sequence 
$(\phi_n)_{n\in\mathbb{N}}\subset \mathcal{K}_{\alpha}$ with $\phi_n\to \phi\;\;\; \mbox{in}\;\;\; \mathcal{C}^{\beta, 1+\beta}(\td\times[0,T])$, we have $\Phi_n:=\mathcal{F}(\phi_n)\to\mathcal{F}(\phi),$ in $\mathcal{C}^{\beta, 1+\beta}(\td\times[0,T])$. We split the proof in two steps.
\bigskip

\noindent{{\bf Step 1}} {\it Stability property of the Fokker-Planck equation:}
 Let $(\rho_n)_{n\in\mathbb{N}}$ be such that
	\[
	\frac{\partial \rho_n}{\partial t}=\div\left( (\nabla W\star \rho_n)\rho_n+\nabla \phi_n\rho_n\right)+\Delta \rho_n\;\;\;\;\;\mbox{in}\;\;\;\;\;\td\times(0,T),
	\]
under the initial condition $\rho_n(x,0)\,=\,\rho_0(x).$ We will prove that $\rho_n\to\rho$ in $L^\infty(0,T; L^2(\Tt^d))$.
We have $$\|\phi_n-\phi\|_{L^\infty(\td\times[0,T])}\to 0, \;\;\;\|\nabla \phi_n\|_{L^\infty(\td\times[0,T])}\leq Ae^{BT}\;\;\; \mbox{and}\;\;\;\left\|\nabla \phi_n(\cdot,t)\right\|_{\mathcal{C}^{\alpha}(\td)}\leq C;$$ hence, we can conclude that $\nabla \phi_n(\cdot, t)\to \nabla \phi (\cdot, t),$ uniformly in $\Tt^d$. In particular, by the Lebesgue's Dominated Convergence Theorem
	\[
	\int_0^T\int_{\td}\rho^2(\nabla\phi_n-\nabla\phi)^2dxdt\to 0.
	\]
	Now, denote $m^n=\rho_n-\rho$; then
	\[
	m^n_t-\div((\nabla W\star \rho)m^n)-\div((\nabla W\star m^n)\rho)-\div(\nabla\phi_n m^n) -\div\left(\rho(\nabla\phi_n-\nabla\phi)\right)=\Delta m^n,
	\]
	under the homogeneous initial condition $ m_n(x,0)=0.$ Proceeding analogously to the previous section, we get
\begin{align*}
	\frac{d}{dt}{\|m^n(\cdot,t)\|_{L^2(\Tt^d)}^2}
	=&\,2\int_{\Tt^d}\left[-|\nabla m^n|^2-m^n\nabla m^n\cdot \nabla\phi_n-\rho \nabla m^n\cdot (\nabla\phi_n-\nabla\phi)\right] dx\\
	&+ 2\int_{\Tt^d}\left[(\nabla W\star \rho)m^n \nabla m^n +(\nabla W\star m^n)\nabla m^n\rho\right]dx\,.
\end{align*}
	Therefore, using the inequality $a\,b\leq\, \epsi a^2\,+\,\frac{b^2}{4\epsi}$, with $\epsi$ small enough and $a\,=\,\nabla m^n$, on each term above starting from the second one , we obtain
	\[
	\frac{d}{dt}\|m^n(\cdot,t)\|_{L^2(\Tt^d)}^2\leq C(\|\nabla\phi_n\|_{\infty}^2+\|\nabla W\|_{L^2}^2\|\rho\|^2_{L^2})\|m^n(\cdot,t)\|_{L^2(\Tt^d)}^2+C\|\rho (\nabla\phi_n-\nabla\phi)\|_{L^2(\Tt^d)}^2.
	\]
	Thus, by Gronwall's inequality, we conclude
	\[
	\|m^n(\cdot,t)\|_{L^2(\Tt^d)}^2\leq \|\rho (\nabla\phi_n-\nabla\phi)\|_{L^2([0,t]\times \Tt^d)}^2e^{C(T-t)}\to 0.
	\]
As a result, $$\rho_n\to \rho\;\;\;\mbox{in}\;\;\;L^{\infty}([0,T], L^2(\Tt^d)),$$as desired.

\bigskip

\noindent{\bf Step 2} {\it Stability and Uniqueness of viscosity solutions for the Hamilton-Jacobi equation:} We now examine the convergence $\Phi_n\to \Phi$. 
We first observe that
	\[
	\|(\nabla W\star \rho_n)-(\nabla W\star \rho)\|_{\infty}\leq \|\nabla W\|_{L^2}\|\rho_n-\rho\|_{L^2}\to 0
	\]
and
	\[
	\|U(\cdot, \rho_n)-U(\cdot, \rho)\|_{\infty}\leq C\|\rho_n-\rho\|_{L^2}\to 0.
	\]
Since $\nabla \phi_n(\cdot, t)\to \nabla \phi (\cdot, t)$ uniformly due to our functional setting, then
$$
\nabla W\star \rho \nabla \phi_n\longrightarrow \nabla W\star \rho \nabla \phi
$$ 
uniformly on $[0,T]\times\Tt^d.$ Hence, the nonlocal terms
	\begin{align*}
	|\nabla W\star(\rho_n\nabla \phi_n)-\nabla W\star(\rho\nabla \phi)|\,\leq &\, \|\nabla\phi_n\|_{\infty}\|\nabla W\|_{L^2}\|\rho_n-\rho\|_{L^2}\\
	&+\,|\nabla W\star\rho (\nabla\phi_n-\nabla\phi)|
	\end{align*}
converge to zero uniformly.
By the compactness of $\mathcal{K}_{\alpha}$ there is a converging subsequence $\Phi_{n_k}\to\tilde{\Phi}$, in $\mathcal{C}^{\beta, 1+\beta}(\td\times[0,T])$, in particular $\nabla \Phi_{n_k}\to\nabla\tilde{\Phi}$ uniformly. Then $\tilde{\Phi}$ is a weak solution to
	\[
	-{\tilde\Phi}_t+\frac{|\nabla{\tilde\Phi}|^2}{2}+(\nabla W\star \rho)\cdot \nabla {\tilde\Phi}+\nabla W\star (\rho \nabla \phi)=\Delta {\tilde\Phi} +U(x,\rho),\,\qquad {\tilde\Phi}(x,T)=\phi_T(x).
	\]
By parabolic regularity $\tilde\Phi$ is also a classical solution. Thus $\Phi$ and $\tilde\Phi$ solve the same equation, to prove that $\Phi=\tilde{\Phi},$ we introduce two adjoint variables $\eta, \tilde{\eta}$ which solve the following equations respectively
	\begin{equation*}
	\begin{cases}
	\frac{\partial \eta}{\partial t}=\div\left( (\nabla W\star \rho)\eta+\nabla \Phi\eta\right)+\Delta \eta,\\
	\eta(\cdot,0)=\delta_{x_0},
	\end{cases}
	\end{equation*}
and
	\begin{equation*}
	\begin{cases}
	\frac{\partial \tilde\eta}{\partial t}=\div\left( (\nabla W\star \rho)\tilde\eta+\nabla \tilde\Phi\tilde\eta\right)+\Delta \tilde\eta,\\
	\tilde\eta(\cdot,0)=\delta_{x_0}.
	\end{cases}
	\end{equation*}
By subtracting the equation for $\Phi$ from the one for $\tilde{ \Phi}$, we deduce
\[
-\left({\tilde\Phi}-\Phi\right)_t+\frac{|\nabla{\tilde\Phi}|^2}{2}-\frac{|\nabla{\tilde\Phi}|^2}{2}+(\nabla W\star \rho)\cdot \nabla \left({\tilde\Phi}- \Phi\right)=\Delta \left({\tilde\Phi}-\Phi\right),
\]
and similarly,
\[
(\tilde{\eta}-\eta)_t=\div\left( (\nabla W\star \rho)(\tilde{\eta}-\eta)+\nabla \tilde\Phi\tilde\eta-\nabla \Phi\eta\right)+\Delta (\tilde\eta-\eta)\,.
\]
Multiplying the first equation above by $\tilde{\eta}-\eta$ and subtracting the second equation multiplied by ${\tilde\Phi}-\Phi$ and then integrating by parts on $[0,T]\times\Tt^d$, we obtain
\[
\int_{\td}\!\!\!({\tilde\Phi}-\Phi)({\tilde\eta}-\eta)dx\big|_0^T=\int_0^T\!\!\!\int_{\td}\!\!\left[  \left(\frac{|\nabla{\tilde\Phi}|^2}{2}-\frac{|\nabla{\tilde\Phi}|^2}{2}\right)\!({\tilde\eta}-\eta) -\left(\nabla \tilde\Phi\tilde\eta-\nabla \Phi\eta\right) \!\cdot\! \nabla({\tilde\Phi}-\Phi) \right]\!\!dx dt.
\]
Recalling the terminal conditions on $\Phi$, $\tilde\Phi$ and initial conditions on $\eta, \tilde{\eta}$, we have that the left hand side of the equality is zero, and thus
\[
\int_0^T\!\!\!\int_{\td}\frac{|\nabla{\tilde\Phi} -\nabla{\Phi} |^2}{2}({\tilde\eta}+\eta)dx dt=0.
\]
Hence $\nabla{\tilde\Phi} = \nabla{\Phi}$ and then $\Phi _t=\tilde{ \Phi}_t$. Since $\Phi (T, x)=\tilde{ \Phi}(T, x)=\phi_T(x)$, we obtain  $\Phi =\tilde{ \Phi}$. The above arguments show that any subsequence of the sequence $\{\Phi_n\}_{n\geq 1}$ has a further subsequence converging to $\Phi$ in  $\mathcal{C}^{\beta, 1+\beta}(\td\times[0,T])$. This proves that $\Phi_n\to \Phi$ in $\mathcal{C}^{\beta, 1+\beta}(\td\times[0,T])$.
Thus we have proved that the conditions of Schauder's Fixed-Point Theorem are satisfied, we infer that $\mathcal{F}$ has a fixed point in $\mathcal{K}_{\alpha}$ and the proof is complete.
\end{proof}

{\small {\bf Acknowledgements.-}
JAC was partially supported by the EPSRC grant number EP/P031587/1.
EAP was partially supported by FAPERJ (\# E26/200.002/2018), CNPq-Brazil (\#433623/2018-7 and \#307500/2017-9) and Instituto Serrapilheira.
We would like to acknowledge the Institute Mittag-Leffler, Imperial College London and King Abdullah University of Science and Technology for hosting us and providing with constant help and vivid research environment.}

\bibliographystyle{siam}\bibliography{mfg} 

\end{document}